\documentclass[leqno,12pt]{amsart} %leqno is the option to put formula numbers on the left side
\usepackage[top=30truemm,bottom=30truemm,left=25truemm,right=25truemm]{geometry}
\usepackage{amssymb}
\usepackage{amsmath}
\usepackage{amsthm}
\usepackage{amscd}
\usepackage{mathrsfs}
\usepackage{graphicx}
\usepackage[dvips]{color}
\usepackage[all]{xy}

\usepackage{url}
\usepackage{comment}
  \makeatletter

\@addtoreset{equation}{section}
\makeatother
\theoremstyle{plain} %text of this environment is typesetted in italics
\newtheorem{theorem}{\indent\bf Theorem}[section]
\newtheorem{lemma}[theorem]{\indent\bf Lemma}

\theoremstyle{definition} %text of this environment is typesetted in roman letters
\newtheorem{definition}[theorem]{\indent\bf Definition}
\newtheorem{remark}[theorem]{\indent\bf Remark}
\newtheorem{example}[theorem]{\indent\bf Example}

\newtheorem{question}[theorem]{\indent\bf Question}
\newcommand{\ddbar}{\partial \bar{\partial}}
\newcommand{\dbar}{\bar{\partial}}

\begin{document}
\pagestyle{plain}
\thispagestyle{plain}

\title[On converse of H\"ormander estimate]
{A converse of H\"ormander's $L^2$-estimate and new positivity notions for vector bundles}

\author[Genki HOSONO and Takahiro INAYAMA]{Genki HOSONO$^{1}$ and Takahiro INAYAMA$^{2}$}
\address{ % First Author
$^{1}$  Graduate School of Mathematical Sciences\\
The University of Tokyo\\
3-8-1 Komaba, Meguro-ku\\
Tokyo, 153-8914\\
Japan
}
\email{genkih@ms.u-tokyo.ac.jp}
\address{% Second Author
$^{2}$  Graduate School of Mathematical Sciences\\
The University of Tokyo\\
3-8-1 Komaba, Meguro-ku\\
Tokyo, 153-8914\\
Japan
}
\email{inayama@ms.u-tokyo.ac.jp}
%\email{e-mail}
%\subjclass[2010]{32J25, 14C20}
%\date{\today}

%%%%%%%%%%%%%%%%%%%%%%%%%%%%%%%%%%%%%%%%%%%%%%%%%%%%%%
\begin{abstract}
We study conditions of H\"ormander's $L^2$-estimate and the Ohsawa-Takegoshi extension theorem. 
Introducing a twisted version of H\"ormander-type condition, we show a converse of H\"ormander $L^2$-estimate under some regularity assumptions on an $n$-dimensional domain.
This result is a partial generalization of the 1-dimensional result obtained by Berndtsson. 
We also define new positivity notions for vector bundles with singular Hermitian metrics by using these conditions. 
We investigate these positivity notions and compare them with classical positivity notions. 
\end{abstract}

%%%%%%%%%%%%%%%%%%%%%%%%%%%%%%%%%%%%%%%%%%%%%%%%%%%%%%

\maketitle

%%%%%%%%%%%%%%%%%%%%%%%%%%%%%%%%%%%%%%%%%%%%%%%%%%%%%%

\section{Introduction}
There are many curvature-positivity notions for Hermitian holomorphic vector bundles on complex manifolds. The situation is simple for line bundles: a Hermitian metric $h=e^{-\phi}$ on a line bundle has positive curvature if and only if the corresponding local weight $\phi$ is plurisubharmonic. When we consider a singular Hermitian metric, its curvature, defined as a current, is positive.

For general vector bundles, the situation is much complicated. There are some positivity notions which are not equivalent to each other.
When we consider a singular metric on vector bundles, we cannot even define its curvature (\cite[Theorem 1.3]{Rau1}).
Therefore, to generalize curvature-positivity concepts of vector bundles for singular metrics, we have to seek their characterizations without using curvature tensors.

For the Griffiths semi-positivity, such a characterization has been obtained (cf. \cite{HPS}, \cite{PT}, \cite{Rau1}). 
On the other hand, we do not know such a characterization for the Nakano semi-positivity, which is stronger than the Griffiths positivity and a natural setting for using H\"ormander's $L^2$-methods. This is one of the difficulties in the study of singular Hermitian metrics on vector bundles.

In \cite{DWZZ}, a new characterization of a plurisubharmonicity was obtained. Namely, the possibility of the $L^2$-extension with a certain condition on its estimate is equivalent to the plurisubharmonicity.
The precise condition is as follows:
\begin{definition}[\cite{DWZZ}]\label{def:ohsawa-takegoshi-condition}
	Let $\Omega \subset \mathbb{C}^n$ be a domain and $\phi$ be an upper semi-continuous function.
	We say that $(\Omega, \phi)$ satisfies {\it the multiple $L^2$-extension property} if there exists a number $C_m > 0$ for each $m$ such that
	\begin{itemize}
		\item for every $p \in \Omega$ with $\phi(p) \neq -\infty$, there exists a holomorphic function $f$ satisfying $f(p) = 1$ and
		$$\int_\Omega |f|^2 e^{-m\phi} \leq C_m e^{-m\phi(0)}, $$ and 
		\item (Growth condition) $C_m$ satisfies the condition $\lim_{m \to +\infty} \frac{\log C_m}{m} = 0$. 
	\end{itemize}
\end{definition}

For vector bundles, it is proved in \cite{DWZZ} that the multiple $L^2$-extension property implies the Griffiths positivity.
In \cite{Ber}, it is proved that, for a continuous function on a one-dimensional domain, the availability of the H\"ormander estimate implies the subharmonicity.
In \cite{DWZZ}, Deng, Wang, Zhang, and Zhou asked a problem if one can extend this result to higher dimensional cases. 
In Section \ref{sec:horlinebundle}, we show a partial converse of H\"ormander's $L^2$-estimate for line bundles on an $n$-dimensional domain.
We define the twisted H\"ormander condition as follows:
\begin{definition}
	We say that $(\Omega, \phi)$ satisfies {\it the twisted H\"ormander condition} if, for every positive integer $m$, smooth strictly plurisubharmonic function $\psi$ on $\Omega$, and smooth $\dbar$-closed $(0,1)$-form $\alpha$ with compact support and finite norm $\int_\Omega |\alpha|^2_{\sqrt{-1}\ddbar\psi} e^{-(m\phi + \psi)}<+\infty$, there exists a smooth function $u$ such that
	\begin{itemize}
		\item $\dbar u = \alpha$, and
		\item $\displaystyle\int_\Omega |u|^2 e^{-(m\phi + \psi)}  \leq \displaystyle\int_\Omega |\alpha|^2_{\sqrt{-1}\ddbar\psi} e^{-(m\phi + \psi)}.$
	\end{itemize}
\end{definition}
Using this, we state a converse of H\"ormander's $L^2$-estimate as follows: 
\begin{theorem}\label{mainthm:hormander-converse}
Let $\Omega \subset \mathbb{C}^n$ be a domain. Assume that ${\rm Pole}(\phi)$ is closed and $\phi$ is a locally H\"older continuous function on $\Omega \setminus {\rm Pole}(\phi)$. 
	If the twisted H\"ormander condition is satisfied for $(\Omega, \phi)$, 
	$\phi$ is plurisubharmonic. 
\end{theorem}
Here we consider a {\it twisted} version of H\"ormander conditions.
Twisting with an additional weight $\psi$ enables us to prove the multiple $L^2$-extension property and thus the plurisubharmonicity of $\phi$.

We shall also prove a partial converse of H\"ormander's $L^2$-estimate for vector bundles (see Theorem \ref{thm:conversevectorbundle}). 
We also introduce several positivity notions for vector bundles. 
To be precise, we study a singular Hermitian metric which is positive in the sense of twisted H\"ormander (see Definition \ref{def:hormanderpositive}) and a singular Hermitian metric 
which is positive in the sense of weak Ohsawa-Takegoshi (see Definition \ref{def:otpositive}). 

It is important to consider the sheaf of square integrable holomorphic sections of vector bundles with respect to positively curved singular Hermitian metrics.
For line bundles, such a sheaf is called {\it the multiplier ideal sheaf}. Multiplier ideal sheaves are proved to be coherent by $L^2$-estimates.
For general vector bundles, little is known about the coherence of such sheaves, because of the lack of $L^2$-estimates for general singular Hermitian metrics.
The first author has proved the coherence of such sheaves for singular Hermitian metrics induced by holomorphic sections \cite[Theorem 1.1]{Hos}. 
We prove that the sheaf of locally square integrable holomorphic sections with respect a metric which is positive in the sense of twisted H\"ormander is coherent. 
\begin{theorem}\label{mainthm:coherent}
Let $(E, h)$ be a positively curved metric in the sense of twisted H\"ormander. Assume that $|u|_{h^\star}$ is upper semi-continuous for any local holomorphic section $u\in \mathscr{O}(E^\star)$.
Then $\mathscr{E}(h)$ is a coherent subsheaf of $\mathscr{O}(E)$, 
where $\mathscr{E}(h)$ is the sheaf of locally square integrable holomorphic sections of $E$ with respect to $h$. 
\end{theorem}

We also study metrics which is positive in the sense of weak Ohsawa-Takegoshi. We show that the positivity in this sense is strictly weaker than the Nakano semi-positivity. 

\begin{theorem}\label{mainthm:notnakanopositive}
There exists a positively curved vector bundle $(E, h)$ in the sense of weak Ohsawa-Takegoshi such that 
$(E, h)$ is not Nakano semi-positive. 
\end{theorem}

This is a partial answer to the question proposed by Deng, Wang, Zhang, and Zhou in \cite{DWZZ} (see Question \ref{qus:griffithsnakano} below).
We also propose some questions about the above new positivity notions. 

\section{A converse of H\"ormander's $L^2$ estimate for line bundles}\label{sec:horlinebundle}
In this section, we formulate a H\"ormander-type condition and prove the equivalence to plurisubharmonicity under some regularity assumption.

\begin{definition}\label{def:hormander-condition}
	Let $\Omega \subset \mathbb{C}^n$ be a domain and $\phi : \Omega \to [-\infty, +\infty)$ be an upper semi-continuous function.
	We say that $(\Omega, \phi)$ satisfies {\it the twisted H\"ormander condition} if, for every positive integer $m$, smooth strictly plurisubharmonic function $\psi$ on $\Omega$, and smooth $\dbar$-closed $(0,1)$-form $\alpha$ with compact support and finite norm $\int_\Omega |\alpha|^2_{\sqrt{-1}\ddbar\psi} e^{-(m\phi + \psi)}<+\infty$, there exists a smooth function $u$ such that
	\begin{itemize}
		\item $\dbar u = \alpha$, and
		\item $\displaystyle\int_\Omega |u|^2 e^{-(m\phi + \psi)}  \leq \displaystyle\int_\Omega |\alpha|^2_{\sqrt{-1}\ddbar\psi} e^{-(m\phi + \psi)}.$
	\end{itemize}
\end{definition}

\begin{remark}
(1) To formulate the condition, we used an additional weight function $\psi$. This enables us to prove the multiple $L^2$-extension property (under some regularity assumption) and therefore the plurisubharmonicity of $\phi$.

(2) We clearly see that, on a domain in $\mathbb{C}^n$, a upper semi-continuous function $\phi$ is plurisubharmonic if and only if for every smooth strictly plurisubharmonic function $\psi$, $\phi + \psi$ is plurisubharmonic. 
Hence, it is worth considering a twisted version of H\"ormander's $L^2$ estimate. 

(3) If $\phi$ is a plurisubharmonic, the twisted H\"ormander condition is automatically satisfied. Indeed, the standard estimate gives the following:
$$\int_\Omega |u|^2 e^{-(m\phi + \psi)} \leq \int_\Omega |\alpha|^2_{\sqrt{-1}\ddbar (m\phi + \psi)} e^{-(m\phi + \psi)}. $$
Since $\phi$ is plurisubharmonic, we have that $\sqrt{-1}\ddbar (m\phi + \psi) \geq \sqrt{-1} \ddbar\psi$ (in the sense of currents). Then we have that $|\cdot|^2_{\sqrt{-1}\ddbar (m\phi + \psi)} \leq |\cdot|^2_{\sqrt{-1}\ddbar \psi}$ and thus $$\int_\Omega |u|^2 e^{-(m\phi + \psi)} \leq \int_\Omega |\alpha|^2_{\sqrt{-1}\ddbar (m\phi + \psi)} e^{-(m\phi + \psi)} \leq \int_\Omega |\alpha|^2_{\sqrt{-1}\ddbar \psi} e^{-(m\phi + \psi)}.$$

(4) One may formulate, as in Definition \ref{def:ohsawa-takegoshi-condition}, the twisted H\"ormander condition with constants $C_m$ and the same growth condition $\lim_{m\to \infty} \frac{\log C_m}{m} =0$.
\end{remark}

We will show that, under some continuity assumption, the H\"ormander condition implies the multiple $L^2$-extension property. 
We set ${\rm Pole}(\phi):=\{ \phi^{-1}(-\infty)\}$, the set of poles of $\phi$.

\begin{theorem}\label{thm:hormander-converse}{\rm (= Theorem \ref{mainthm:hormander-converse})}
	Let $\Omega \subset \mathbb{C}^n$ be a domain. Assume that ${\rm Pole}(\phi)$ is closed and $\phi$ is a locally H\"older continuous function on $\Omega \setminus {\rm Pole}(\phi)$, i.e.\ for every $\Omega' \Subset \Omega \setminus {\rm Pole}(\phi)$, there exist constants $\alpha = \alpha_{\Omega'} \in (0,1]$ and $C = C_{\Omega'}>0$ such that $|\phi(z) - \phi(w)| \leq C |z-w|^\alpha$ for every $z,w \in \Omega'$.
	If the twisted H\"ormander condition (Definition \ref{def:hormander-condition}) is satisfied for $(\Omega, \phi)$, 
	$\phi$ is plurisubharmonic.
\end{theorem}

\begin{proof}
Fix a domain $\Omega' \Subset \Omega \setminus {\rm Pole}(\phi)$.
We will show that $(\Omega', \phi)$ satisfies the multiple $L^2$-extension property (Definition \ref{def:ohsawa-takegoshi-condition}). 

Fix a point $w \in \Omega'$ with $\phi(w)>-\infty $ and an integer $m>0$. We will construct a holomorphic function $f \in \Omega'$ such that $f(w) = 1$ and
$$ \int_{\Omega'} |f|^2 e^{-m\phi} \leq C_m e^{-m\phi(w)},$$
where $C_m$ is a constant independent of the choice of $w \in \Omega'$.
%Note that, by assumption that $\phi$ is H\"older continuous, we have that $\phi(w) \neq -\infty$.
%In the following, we assume $w = 0$ for the simplicity.

Take $\chi = \chi(t)$ be a smooth function on $\mathbb{R}$, such that
\begin{itemize}
	\item $\chi(t) = 1$ for $t \leq 1/2$,
	\item $\chi(t) = 0$ for $t \geq 1$, and
	\item $|\chi'(t)|\leq 3$ on $\mathbb{R}$.
\end{itemize}
Define a (0,1)-form $\alpha$ by
\begin{align*}
\alpha &:= \dbar \chi\left({|z-w|^2 \over \epsilon^2 } \right)= \chi'\left(\frac{|z-w|^2}{\epsilon^2} \right) \sum_j \frac{z_j-w_j}{\epsilon^2} d\overline{z}_j.
\end{align*}
We apply the twisted H\"ormander condition for the weight
\begin{equation*}
\psi_{\epsilon, \delta} := \log(|z-w|^2 + \epsilon^2) + n \log(|z-w|^2 + \delta^2),
\end{equation*}
where $\epsilon$ and $\delta$ are positive parameters. 

Then we obtain a smooth function $u_{\epsilon, \delta}$ on $\Omega'$ such that
\begin{itemize}
	\item $\dbar u_{\epsilon, \delta} = \alpha$ and
	\item \begin{equation}
	\int_{\Omega'} |u_{\epsilon, \delta}|^2 e^{-(m\phi + \psi_{\epsilon, \delta})} \leq \int_{\Omega'} |\alpha|^2_{\sqrt{-1}\ddbar \psi_{\epsilon, \delta}} e^{-(m\phi+ \psi_{\epsilon, \delta})}.\label{eqn:dbar-est}
	\end{equation}
\end{itemize}
Since
\begin{equation*}
|\alpha|^2_{\sqrt{-1}\ddbar \psi_{\epsilon, \delta}} = \left| \chi' \left(\frac{|z-w|^2}{\epsilon^2} \right) \right|^2 \cdot \frac{1}{\epsilon^4} \cdot \left|\sum_j (z_j-w_j) d\overline{z}_j \right|^2_{\sqrt{-1}\ddbar \psi_{\epsilon, \delta}}
\end{equation*}
and the support of $\chi'\left(\frac{|z-w|^2}{\epsilon^2} \right)$ is in $\{1/2 \leq |z-w|^2 /\epsilon^2 \leq 1 \}$, we have that
\begin{align}
(\text{RHS of (\ref{eqn:dbar-est})}) &= \int_{\{1/2 \leq |z-w|^2/\epsilon^2 \leq 1 \}} \left| \chi' \left(\frac{|z-w|^2}{\epsilon^2} \right) \right|^2 \cdot \frac{1}{\epsilon^4} \cdot \left|\sum_j (z_j-w_j) d\overline{z}_j \right|^2_{\sqrt{-1}\ddbar \psi_{\epsilon, \delta}} e^{-(m\phi + \psi_{\epsilon, \delta})}\\
& \leq \frac{9}{\epsilon^4} \int_{\{\epsilon^2 / 2 \leq |z-w|^2 \leq \epsilon^2\}} \left|\sum_j (z_j-w_j) d\overline{z}_j \right|^2_{\sqrt{-1}\ddbar \psi_{\epsilon, \delta}} e^{-(m\phi + \psi_{\epsilon, \delta})}.\label{eqn:est-1}
\end{align}
Letting $\omega := i\ddbar |z|^2$, we have that 
\begin{equation*}
\sqrt{-1}\ddbar \psi_{\epsilon, \delta} \geq \sqrt{-1}\ddbar \log(|z-w|^2+\epsilon^2) \geq \frac{\epsilon^2}{(|z-w|^2+\epsilon^2)^2} \omega,
\end{equation*}
and thus 
\begin{equation}
|\cdot|^2_{\sqrt{-1}\ddbar \psi_{\epsilon, \delta}} \leq |\cdot|^2_{(\epsilon^2/(|z-w|^2+\epsilon^2)^2)\omega}.\label{eqn:norm_est}
\end{equation}
Combining (\ref{eqn:est-1}) and (\ref{eqn:norm_est}), we have the following:
\begin{align*}
(\ref{eqn:est-1}) &\leq \frac{9}{\epsilon^4} \int_{\{\epsilon^2 / 2 \leq |z-w|^2 \leq \epsilon^2\}} |z-w|^2 \frac{(|z-w|^2 + \epsilon^2)^2}{\epsilon^2} e^{-(m\phi + \psi_{\epsilon, \delta})}\\
&= \frac{9}{\epsilon^4} \int_{\{\epsilon^2 / 2 \leq |z-w|^2 \leq \epsilon^2\}} |z-w|^2 \frac{(|z-w|^2 + \epsilon^2)^2}{\epsilon^2} \frac{1}{|z-w|^2 + \epsilon^2}\frac{1}{(|z-w|^2 + \delta^2)^n}e^{-m\phi}\\
&= \frac{9}{\epsilon^4} \int_{\{\epsilon^2 / 2 \leq |z-w|^2 \leq \epsilon^2\}} |z-w|^2 \frac{|z-w|^2 + \epsilon^2}{\epsilon^2} \frac{1}{(|z-w|^2 + \delta^2)^n}e^{-m\phi}\\
& \leq \frac{9}{\epsilon^4} {\rm Vol}(\{\epsilon^2 / 2 \leq |z-w|^2 \leq \epsilon^2\}) \frac{\epsilon^2(\epsilon^2 + \epsilon^2)}{\epsilon^2} \frac{1}{(\epsilon^2/2 + \delta^2)^n} e^{-m\inf_{B(w,\epsilon)} \phi}\\
&= \frac{9}{\epsilon^4} C_n \epsilon^{2n} \frac{\epsilon^2(\epsilon^2 + \epsilon^2)}{\epsilon^2} \frac{1}{(\epsilon^2/2 + \delta^2)^n} e^{-m\inf_{B(w,\epsilon)} \phi}\\
&= 9C_n \epsilon^{2n-2} \frac{1}{(\epsilon^2/2 + \delta^2)^n}e^{-m\inf_{B(w,\epsilon)} \phi}.
\end{align*}

To summarize, we have obtained a smooth function $u_{\epsilon, \delta}$ on $\Omega'$ such that
\begin{itemize}
	\item $\dbar u_{\epsilon, \delta} = \alpha$, and
	\item the following estimate holds:
	\begin{equation}
	\int_{\Omega'} |u_{\epsilon, \delta}|^2 e^{-(m\phi + \psi_{\epsilon, \delta})} \leq 9C_n \epsilon^{2n-2} \frac{1}{(\epsilon^2/2 + \delta^2)^n}e^{-m\inf_{B(w,\epsilon)} \phi}.\label{eqn:matome}
	\end{equation}
\end{itemize}

Let $\delta \to 0$. The right-hand side of (\ref{eqn:matome}) is increasing to 
\begin{equation*}
9C_n \epsilon^{2n-2} \frac{2^n}{\epsilon^{2n}}e^{-m\inf_{B(w,\epsilon)} \phi} = 9C_n \frac{2^n}{\epsilon^2} e^{-m\inf_{B(w,\epsilon)} \phi}.
\end{equation*}
Let us show the convergence of $u_{\epsilon, \delta}$.
%ここは「By the standard argument」で済ませてしまってもいいと思うが…
We have that
$$\int_{\Omega'} |u_{\epsilon, \delta}|^2 e^{-(m\phi + \psi_{\epsilon, \delta})} \leq 9C_n \frac{2^n}{\epsilon^2} e^{-m \inf \phi}.  $$
Note that the weight function $\psi_{\epsilon, \delta}$ is decreasing when $\delta \to 0$. Therefore $e^{-\psi_{\epsilon, \delta}}$ is increasing when $\delta \to 0$.

Fix $\delta_0>0$. Then, for $\delta < \delta_0$, (since $e^{-\psi_{\epsilon, \delta}} > e^{-\psi_{\epsilon, \delta_0}} $) we have that
$$\int_{\Omega'} |u_{\epsilon, \delta}|^2 e^{-(m\phi + \psi_{\epsilon, \delta_0})} \leq \int_{\Omega'} |u_{\epsilon, \delta}|^2 e^{-(m\phi + \psi_{\epsilon, \delta})} \leq 9C_n \frac{2^n}{\epsilon^2} e^{-m \inf \phi}.  $$
Thus $\{u_{\epsilon, \delta}\}_{\delta < \delta_0}$ forms a bounded sequence in $L^2(\Omega', e^{-(m\phi + \delta_0)})$. We can choose a weakly convergent sequence $\{u_{\epsilon, \delta^{(k)}}\}_k$ in $L^2(\Omega', e^{-(m\phi + \delta_0)})$. Since the $L^2$-norm is lower semi-continuous under weak limits, the (weak) limit function $u_\epsilon$ satisfies
$$\int_{\Omega'} |u_{\epsilon}|^2 e^{-(m\phi + \psi_{\epsilon, \delta_0})} \leq  9C_n \frac{2^n}{\epsilon^2} e^{-m \inf \phi}. $$
Next, fix $\delta_1 < \delta_0$. By the same argument, we can choose a subsequence of $\{u_{\epsilon, \delta^{(k)}}\}_k$ weakly convergent in $L^2(\Omega', e^{-(m\phi + \delta_1)})$.
Repeating this argument for a sequence $\{\delta_j\}$ decreasing to $0$ and taking a diagonal sequence, finally we can obtain a sequence $\{u_{\epsilon, \delta^{(k)}}\}_k$ weakly convergent in $L^2(\Omega', e^{-(m\phi + \psi_{\epsilon, \delta_j})})$ for every $j$. Then we have 
$$\int_{\Omega'} |u_{\epsilon}|^2 e^{-(m\phi + \psi_{\epsilon, \delta_j})} \leq  9C_n \frac{2^n}{\epsilon^2} e^{-m \inf \phi} $$
for every $j$ and, by the monotone convergence theorem,
$$\int_{\Omega'} |u_{\epsilon}|^2 e^{-(m\phi + \psi_{\epsilon})} \leq  9C_n \frac{2^n}{\epsilon^2} e^{-m \inf \phi}. $$

Since differential operators are continuous under weak limits, we have $\dbar u_{\epsilon} = \alpha.$

The integral
$$\int_{\Omega'} |u_\epsilon|^2 e^{-(m\phi + \psi_\epsilon)} = \int_{\Omega'} |u_\epsilon|^2 e^{-m\phi} \frac{1}{(|z-w|^2 + \epsilon^2)} \frac{1}{|z-w|^{2n}} $$
is finite while the weight $\frac{1}{|z-w|^{2n}}$ is not integrable near $w$, thus $u_\epsilon(w)$ must be $0$.

Let $f_\epsilon := \chi(|z-w|^2/\epsilon^2) - u_\epsilon$. Then $f_\epsilon(0) = 1$ and
\begin{equation}
\left(\int_{\Omega'} |f_\epsilon|^2 e^{-m\phi}\right)^{1/2} \leq \left(\int_{\Omega'} \left|\chi\left({|z-w|^2\over\epsilon^2}\right)\right|^2 e^{-m\phi}\right)^{1/2} + \left(\int_{\Omega'} |u_\epsilon|^2 e^{-m\phi}\right)^{1/2}.\label{eqn:triangle}
\end{equation}

We will estimate each term of the right-hand side of (\ref{eqn:triangle}).
Since $\chi \leq 1$ and the support of $\chi(|z-w|^2 / \epsilon^2) $ is contained in $\{|z-w|^2 \leq \epsilon^2 \}$, the first term can be bounded by
$$ e^{2n}e^{-m \inf_{B(w,\epsilon)} \phi}.$$
Next we consider the second term. We have that
\begin{align*}
\int_{\Omega'}|u|^2 e^{-m\phi} &\leq \left[\sup_{z \in \Omega'} |z-w|^{2n}(|z-w|^2+\epsilon)\right] \cdot \int_{\Omega'} |u|^2 e^{-m\phi } \frac{1}{(|z-w|^2 + \epsilon^2)} \frac{1}{|z-w|^{2n}} \\&\leq
\left[\sup_{z \in \Omega'} |z-w|^{2n}(|z-w|^2+\epsilon)\right] \cdot 9C_n \frac{2^n}{\epsilon^2} e^{-m\inf_{B(w,\epsilon)} \phi}.
\end{align*}
Assuming $\epsilon < 1$, we can bound the sup by $(R+1)^{2n+2}$, where $R$ is the radius of $\Omega'$.
Therefore the second term is bounded by
$ C \frac{1}{\epsilon^2} e^{-m\inf_{B(w,\epsilon)} \phi}. $

Therefore, we have that
\begin{equation*}
\int_{\Omega'} |f|^2 e^{-m\phi} \leq C' \left(\epsilon^n + \frac{1}{\epsilon}\right)^2 e^{-m \inf_{B(w,\epsilon)} \phi},
\end{equation*}
where $C'$ is a constant independent of $m$.

By the assumption that $\phi$ is locally H\"older continuous on $\Omega \setminus {\rm Pole}(\phi)$, we have that $|\phi(z) - \phi(w)| \leq C_{\Omega'}|z-w|^\alpha$ for every $z \in \Omega'$.
Let $\epsilon := 1/(mC_{\Omega'})^{1/\alpha}$. Then we have that $|m\phi(z) - m \phi(w)| \leq 1$ for $|z-w| \leq \epsilon$. Thus,
\begin{align}
\int_{\Omega'} |f|^2 e^{-m\phi} &\leq C'\left(\frac{1}{m^{n \over \alpha}C_{\Omega'}^{n \over \alpha}} +{m^{1 \over \alpha}C_{\Omega'}^{1 \over \alpha}}\right)^2 e^{-m \phi(w) + 1}\nonumber\\
&= C'e\cdot \left(\frac{1}{m^{n \over \alpha}C_{\Omega'}^{n \over \alpha}} +{m^{1 \over \alpha}C_{\Omega'}^{1 \over \alpha}}\right)^2e^{-m \phi(w)}.\label{eqn:last}
\end{align} 
The coefficient of (\ref{eqn:last}) satisfies the growth condition $\frac{\log C_m}{m} \to 0$, thus we have verified the multiple $L^2$-extension property for $(\Omega', \phi)$. 

Then, by \cite{DWZZ}, it follows that $\phi$ is plurisubhamronic on $\Omega \setminus {\rm Pole}(\phi)$. Since $\phi$ takes the value $-\infty$ on ${\rm Pole}(\phi)$, $\phi$ 
is also plurisubharmonic on $\Omega$. 
\end{proof}

\begin{remark}
	In this argument, we used the assumption of the regularity of $\phi$. 
	We do not know whether we can completely remove this kind of assumption. 
\end{remark}

\section{A converse of H\"ormander's $L^2$ estimate for vector bundles and relations to various positivity notions}\label{sec:vectorbundle}
\subsection{A converse of H\"ormander's $L^2$ estimate for vector bundles}
In this section, we prove a version of Theorem \ref{thm:hormander-converse} for vector bundles. 
First of all, we introduce the definition of singular Hermitian metrics on vector bundles. 
Throughout this section, $E\to \Omega$ denotes a holomorphic vector bundle over a domain $\Omega \subset \mathbb{C}^n$, 
$\omega$ denotes the standard K\"ahler metric on $\Omega$, $dV_\omega$ be the volume form determined by $\omega$, 
$h$ denotes a singular Hermitian metric on $E$, and $h^{\star}$ denotes the dual metric on 
the dual vector bundle $E^{\star}$.  

\begin{definition}(cf. \cite[Definition 17.1]{HPS}, \cite[Definition 1]{Rau1})
{\it A singular Hermitian metric} $h$ on $E$ is a 
measurable map from $\Omega$ to the space of non-negative Hermitian forms on the fibers, i.e. $h$ satisfies the following conditions:  
\begin{enumerate}
\item $h$ is finite and positive definite almost everywhere on each fiber. 
\item the function $|s|_h : U \to [0, +\infty]$ is measurable, where $U\subset \Omega$ is an open set and $s\in H^0(U, E)$. 
\end{enumerate}
\end{definition}
The Ohsawa-Takegoshi type condition, which is called the multiple $L^p$-extension property, for vector bundles is introduced in \cite{DWZZ}. The precise definition is as folllows. 

\begin{definition}[\cite{DWZZ}](Multiple $L^p$-extension property)\label{def:multipleextensionproperty}
Let $p>0$ be a fixed constant. Assume that for any $z\in \Omega$, any nonzero element $a\in E_z$ with finite norm $|a|_h<+\infty$, and any $m\ge 1$, there is a holomorphic 
section $f_m \in H^0(\Omega, E^{\otimes m})$ such that $f_m(z)=a^{\otimes m}$ and satisfies the following condition:
\begin{equation}
\int_{\Omega} |f_m|^p_{h^{\otimes m}} dV_\omega \le C_m |a^{\otimes m}|^p_{h^{\otimes m}} = C_m |a|^{mp}_{h},\label{eqn:l2extensionproperty}
\end{equation}
where $C_m$ are constants independent of $z\in \Omega$ and satisfy the growth condition $ \lim_{m\to \infty}\frac{1}{m}\log C_m = 0$. 
Then $(E, h)$ is said to have {\it the multiple $L^p$-extension property}. 
\end{definition}

In this paper, we only consider the multiple $L^2$-extension property. We also define the H\"ormander-type condition for vector bundles. 
\begin{definition}\label{def:hormandervector}
	We say that $(E, h)$ satisfies {\it the twisted H\"ormander condition} on $\Omega$ if, for every positive integer $m$, smooth strictly plurisubharmonic function $\psi$ on $\Omega$ and smooth $\dbar$-closed $E^{\otimes m}$-valued $(n,1)$-form $\alpha = \sum_{j} \alpha_j dz \wedge d\overline{z}_j ~(\alpha_j\text{ is a smooth section of }E^{\otimes m} )$ with compact support and finite norm $\int_\Omega \sum_{1\le i, j \le n}(\psi^{(i\overline{j})}\alpha_i, \alpha_j)_{h^{\otimes m}} e^{-\psi}dV_\omega<+\infty $, there exists a smooth $E^{\otimes m}$-valued $(n,0)$ form $u$ such that
	\begin{itemize}
		\item $\dbar u = \alpha$, and
		\item $\displaystyle\int_\Omega |u|^2_{(h^{\otimes m}, \omega)} e^{-\psi}  dV_\omega \leq \displaystyle\int_\Omega \sum_{1\le i, j \le n}(\psi^{(i\overline{j})}\alpha_i, \alpha_j)_{h^{\otimes m}} e^{-\psi}dV_\omega,$
	\end{itemize}
	where $dz:=dz_1 \wedge \ldots \wedge dz_n$ and $(\psi^{(i\overline{j})})_{1\le i,j\le n}$ is the inverse matrix of $(\frac{\partial^2 \psi}{\partial z_i \partial \overline{z}_j})_{1\le i,j\le n}$. 
\end{definition}
\begin{remark}
(1) The above $(\psi^{(i\overline{j})})_{1\le i,j\le n}$ corresponds to the inverse operator of $[\sqrt{-1}\ddbar \psi \otimes Id_{E^{\otimes m}}, \Lambda_{\omega}]$. Here $[\cdot, \cdot]$ denotes a graded Lie bracket, and 
$\Lambda_\omega$ denotes an adjoint of the operator $L:=\omega \wedge \cdot$.

(2) When $h$ is smooth and Nakano positive, we can show that $h$ satisfies the twisted H\"ormander condition.
\end{remark}

We will show that, under some regularity condition, the twisted H\"ormander condition implies the multiple $L^2$-extension property for vector bundles. 
\begin{theorem}\label{thm:conversevectorbundle}
We fix a bounded domain $\Omega' \Subset \Omega$. 
Assume that $|u|_{h^{\star}}$ is upper semi-continuous for any local holomorphic section $u\in \mathscr{O}(E^\star)$. 
Moreover we assume that $\log |u|_h$ is locally H\"older contionuous on $\Omega$ for any non-zero local holomorphic section $u$, 
i.e. for every $\Omega''\Subset \Omega$, $|\log |u(z)|_{h(z)}-\log |u(w)|_{h(w)}|\le C_{\Omega''}|z-w|^\alpha $ for positive constants $\alpha\in (0, 1] $ and 
$C_{\Omega''}>0$ independent of $u\in H^0(\Omega'', E)$. 
If $(E, h)$ satisfies the twisted H\"ormander condition (Definition \ref{def:hormandervector}) on $\Omega$, 
$(E, h)$ has $L^2$-extension property on $\Omega'$. Then, $(E, h)$ is also Griffiths semi-positive. 
\end{theorem}

\begin{proof}
For any $w\in \Omega' $, we take any non-zero element $a\in E_w$. Taking the smooth function $\chi$ in the proof of Theorem \ref{thm:hormander-converse}, 
we define a $\dbar $-closed $E^{\otimes m}$-valued $(n, 1)$-form $\alpha $ by 
\begin{equation*}
\alpha := \dbar \left( a^{\otimes m} \chi\left({|z-w|^2 \over \epsilon^2 } \right) dz \right)= a^{\otimes m} \chi'\left(\frac{|z-w|^2}{\epsilon^2} \right) \sum_j \frac{(z_j-w_j)}{\epsilon^2} dz\wedge d\overline{z}_j. 
\end{equation*}
We also take the weight function 
\begin{equation*}
\psi_{\epsilon, \delta} := \log(|z-w|^2 + \epsilon^2) + n \log(|z-w|^2 + \delta^2)
\end{equation*}
for $\epsilon, \delta >0$. Since
\begin{align*}
\sum_{1\le i, j \le n}\left(\psi_{\epsilon, \delta}^{(i\overline{j})}\frac{(z_i-w_i) \chi'}{\epsilon^2}a^{\otimes m}, \frac{(z_j-w_j) \chi'}{\epsilon^2}a^{\otimes m}\right)_{h^{\otimes m}}&=\frac{|\chi'|^2}{\epsilon^4}|a^{\otimes m}|_{h^{\otimes m}}\sum_{1\le i, j \le n}\psi_{\epsilon, \delta}^{(i\overline{j})}(z_i-w_i) (\overline{z}_j-\overline{w}_j) \\
& \le \frac{|\chi'|^2}{\epsilon^4}|a^{\otimes m}|_{h^{\otimes m}} \frac{|z-w|^2(|z-w|^2+\epsilon^2)^2}{\epsilon^2}, 
\end{align*}
we have 
\begin{align*}
&\displaystyle\int_\Omega \sum_{1\le i, j \le n} \left(\psi_{\epsilon, \delta}^{(i\overline{j})}\frac{(z_i-w_i) \chi'}{\epsilon^2}a^{\otimes m}, \frac{(z_j-w_j) \chi'}{\epsilon^2}a^{\otimes m}\right)_{h^{\otimes m}} e^{-\psi_{\epsilon, \delta}} dV_\omega \\ 
& \le \frac{1}{\epsilon^6} \displaystyle\int_\Omega |\chi'|^2|a^{\otimes m}|_{h^{\otimes m}} |z-w|^2(|z-w|^2+\epsilon^2)^2 e^{-\psi_{\epsilon, \delta}} dV_\omega \\
& \le \frac{9}{\epsilon^6}\displaystyle\int_{\{ \epsilon^2/2 \le |z-w|^2 \le \epsilon^2 \}} \frac{|z-w|^2(|z-w|^2+\epsilon^2)}{(|z-w|^2+\delta^2)^n}|a|^{2m}_{h(z)} dV_\omega \\
& \le 9C_n \epsilon^{2n-2} \frac{1}{(\epsilon^2/2+\delta^2)^n}\sup_{|z-w|\le \epsilon}|a|^{2m}_{h(z)}
\end{align*}
for some positive constant $C_n$ depending only $n$. Repeating the same argument of the proof of Theorem \ref{thm:hormander-converse}, we find a smooth $E^{\otimes m}$-valued $(n,0)$ form $u_\epsilon$ such that 
	\begin{itemize}
		\item $\dbar u_\epsilon = \alpha$, and
		\item $\displaystyle\int_\Omega |u_\epsilon|^2_{(h^{\otimes m}, \omega)} e^{-\psi_{\epsilon}}  dV_\omega \leq  \frac{9C_n 2^n}{\epsilon^2}\sup_{|z-w|\le \epsilon}|a|^{2m}_{h(z)}$. 
	\end{itemize}
	
Since $\log |a|_{h(z)}$ is H\"older continuous, 
\begin{equation*}
\sup_{|z-w|\le \epsilon} |a|^{2m}_{h(z)} \le e^{2mC_{\Omega'}\epsilon^\alpha}|a|^{2m}_{h(w)}
\end{equation*}
for constants $\alpha \in (0, 1]$ and $C_{\Omega'}>0$. 

Let $f_\epsilon := a^{\otimes m}\chi dz - u_\epsilon$. Then $f_\epsilon $ is a holomorphic $(n, 0)$-form, $f_\epsilon(w)=a^{\otimes m}$, and satisfies the following inequality
\begin{equation*}
\displaystyle\int_{\Omega}|f_\epsilon|^2_{(h^{\otimes m, \omega})}dV_\omega \le C \left( \epsilon^{2n}+\frac{1}{\epsilon^2}\right) e^{2mC_{\Omega'}\epsilon^\alpha}|a|^{2m}_{h(w)}, 
\end{equation*} 
where $C>0$ is a constant depending on $\Omega'$. Taking $\epsilon=1/(mC_{\Omega'})^{1/\alpha}$, we have 
\begin{equation*}
\displaystyle\int_{\Omega}|f_\epsilon|^2_{(h^{\otimes m, \omega})}dV_\omega \le e^2C\left( \frac{1}{(mC_{\Omega'})^{\frac{2n}{\alpha}}} + (mC_{\Omega'})^{\frac{2}{\alpha}} \right)|a|^{2m}_{h(w)}. 
\end{equation*}
Since the above coefficient satisfies the growth condition $\frac{\log C_m}{m}\to 0$, we can conclude that $(E, h)$ has $L^2$-extension property on $\Omega'$. 
We also see that $(E, h)$ is Griffiths semi-positive from the results of \cite[Theorem 6.4]{DWZZ}. 
\end{proof}

\begin{remark}
Continuity of $\log h$ implies that $h$ is positive definite and finite on $\Omega$. 
In Theorem \ref{thm:hormander-converse}, $\phi$ has a plurisubharmonic extension from $\Omega \setminus {\rm Pole}(\phi)$ to $\Omega$. 
However, ``the singular set" of singular Hermitian metrics on vector bundles is much more complicated. 
In the case that $h$ has general singularity, we are not sure that Theorem \ref{thm:conversevectorbundle} holds. 
\end{remark}

\subsection{New positivity notions for vector bundles}
In this subsection, we introduce various positivity notions for vector bundles and compare them. Throughout this section, we let 
$(X, \omega)$ be a smooth Hermitian manifold, $E \to X$ be a holomorphic vector bundle of rank $r$ over $X$, 
and $h$ be a singular Hermitian metric on $E$. 

Firstly, 
we define a positively curved singular Hermitian metric in the sense of twisted H\"ormander. 

\begin{definition}\label{def:hormanderpositive}
We say that $(E, h)$ is {\it positively curved in the sense of twisted H\"ormander} if  
for any point $x\in X$, there exists an open neighborhood $U$ of $x$ such that $(E, h)$ satisfies {\it the twisted H\"ormander condition} (cf. Definition \ref{def:hormandervector}) on $U$. 
\end{definition}

We also define a positively curved metric in the sense of Ohsawa-Takegoshi. 

\begin{definition}\label{def:otpositive}
We say that $(E, h)$ is {\it positively curved in the sense of weak Ohsawa-Takegoshi} if  
for any point $x\in X$, there exists an open neighborhood $U$ of $x$ such that $(E, h)$ has {\it the multiple $L^2$-extension property} (cf. Definition \ref{def:multipleextensionproperty}) on $U$. 
\end{definition}

\begin{remark}
The Ohsawa-Takegoshi extension theorem usually states that the constant on the right-hand side of (\ref{eqn:l2extensionproperty}) is independent of $m$. 
Hence, we use the word ``weak Ohsawa-Takegoshi" in Definition \ref{def:otpositive}. 
\end{remark}

Theorem \ref{thm:conversevectorbundle} implies the following theorem. 

\begin{theorem}
Let $(E, h)$ be positively curved in the sense of twisted H\"ormander. If $\log |u|_h$ is locally 
H\"older continuous for any non-zero local holomorphic section $u$, $(E, h)$ is positively curved in the sense of weak Ohsawa-Takegoshi. 
\end{theorem}

In \cite{DWZZ}, the authors proved that a vector bundle which has the multiple $L^2$-extension property is positively curved in the sense of Griffiths. 
Moreover, they proposed the next question. 

\begin{question}\label{qus:griffithsnakano}
Is $L^2$-extension property stronger than Griffiths positivity? 
Is it more or less equivalent to Nakano positivity?  
\end{question}

In this subsection, we give a partial answer to Question \ref{qus:griffithsnakano}. 
To be precise, we show the following theorem. 

\begin{theorem}\label{thm:notnakanopositive}{\rm (= Theorem \ref{mainthm:notnakanopositive})}
There is a positively curved vector bundle $(E, h)$ in the sense of weak Ohsawa-Takegoshi such that 
$(E, h)$ is not Nakano semi-positive. 
\end{theorem}

To prove this theorem, we prepare the following essential lemma. 

\begin{lemma}\label{lemma:quotient}
Let $(E, h)$ be positively curved in the sense of weak Ohsawa-Takegoshi, and $(Q, h_Q)$ be the quotient bundle and metric of $(E, h)$. 
Then, $(Q, h_Q)$ is also positively curved in the sense of weak Ohsawa-Takegoshi. 
\end{lemma}

\begin{proof}%[\indent\sc Proof of Lemma \ref{lemma:quotient}]
Let $\beta : E\to Q$ be the quotient map. 
For any point $x\in X$, there exists an open neighborhood $U$ of $x$ such that $(E, h)$ has the multiple $L^2$-extension property on $U$. 
For any $z\in U$ and a nonzero element $a\in Q_z$ with finite norm $|a|_{h_Q}<+\infty$, there is a nonzero element 
$b\in E_z$ such that $\beta (b) = a$ and $|a|_{h_Q}=|b|_h$. Since $(E, h)$ has the multiple $L^2$-extension property on $U$, for any $m\in \mathbb{N}$, there is a holomorphic 
section $f_m \in H^0(U, E^{\otimes m})$ such that $f_m(z)=b^{\otimes m}$ and 
\begin{equation*}
\int_{U} |f_m|^2_{h^{\otimes m}} dV_\omega \le C_m |b^{\otimes m}|^2_{h^{\otimes m}},
\end{equation*} 
where $C_m$ are constants satisfying the growth condition $\lim_{m\to \infty}\frac{1}{m}\log C_m=0$. Therefore we have 
\begin{align*}
\int_{U} |\beta^{\otimes m}\circ f_m|^2_{{h_Q}^{\otimes m}} dV_\omega & \le \int_{U} |f_m|^2_{{h}^{\otimes m}} dV_\omega \\
& \le C_m |b|^{2m}_{h} \\
& = C_m |a|^{2m}_{h_Q}.
\end{align*}
Consequently, we can conclude that $(Q, h_Q)$ has the multiple $L^2$-extension property on $U$. Hence, $(Q, h_Q)$ is positively curved in the sense of weak Ohsawa-Takegoshi. 
\end{proof}

Here we give the following example. 

\begin{example}\label{ex:notnakanosemipositive}
Let $Q$ be the vector bundle of rank $n$ over $\mathbb{P}^n$ defined by 
\begin{equation*}
0\to \mathcal{O}(-1)\to \underline{\mathbb{C}}^{n+1}\to Q\to 0, 
\end{equation*}
where $\underline{\mathbb{C}}^{n+1}$ is the trivial vector bundle of rank $n+1$ over $\mathbb{P}^n$ and $\mathcal{O}(-1)$ 
be the tautological line bundle. It is known that $Q$ is not Nakano semi-positive (cf. \cite[Chapter VII, Example 6.8]{DemCom}). 
\end{example}

Lemma \ref{lemma:quotient} and Example \ref{ex:notnakanosemipositive} imply 
Theorem \ref{thm:notnakanopositive}. 

\begin{proof}[\indent\sc Proof of Theorem \ref{thm:notnakanopositive}]
We consider the vector bundles $\underline{\mathbb{C}}^{n+1}$ and $Q$ over $\mathbb{P}^n$ in Example \ref{ex:notnakanosemipositive}. Let $h_0$ be the 
standard Euclidean metric on $\underline{\mathbb{C}}^{n+1}$ and $h_q$ be the quotient metric of $h_0$ on $Q$.  
Since $(\underline{\mathbb{C}}^{n+1}, h_0)$ is positively curved in the sense of weak Ohsawa-Takegoshi, 
$(Q, h_q)$ is also positively curved in the sense of weak Ohsawa-Takegoshi. However, 
$(Q, h_q)$ is not Nakano semi-positive. In conclusion, $(Q, h_q)$ satisfies the condition of Theorem \ref{thm:notnakanopositive}. 
\end{proof}

Now we consider the coherence of higher rank analogue of multiplier ideal sheaves.
For a line bundle with a singular Hermitian metric, the sheaf of locally square integrable holomorphic sections is called the multiplier ideal sheaf. It is known that multiplier ideal sheaves are coherent for positively curved singular Hermitian metrics.
For vector bundles, the coherence of these sheaves is not known in general due to the lack of results like $L^2$-estimates.
Here we prove that the twisted H\"ormander condition implies the coherence of the sheaf of square integrable holomorphic sections.

\begin{theorem}\label{thm:coherent}{\rm ($=$ Theorem \ref{mainthm:coherent})}
Let $(E, h)$ be positively curved in the sense of twisted H\"ormander. Assume that $|u|_{h^\star}$ is upper semi-continuous for any local holomorphic section $u\in \mathscr{O}(E^\star)$.
Then $\mathscr{E}(h)$ is a coherent subsheaf of $\mathscr{O}(E)$, 
where $\mathscr{E}(h)$ is the sheaf of locally square integrable holomorphic sections of $E$ with respect to $h$. 
\end{theorem}

\begin{proof}
The following proof is based on the proof of \cite[Proposition 5.7]{DemAna}. 

For any point $x\in X$, there exists an open neighborhood $\Omega$ of $x$ such that $(E, h)$ satisfies the twisted H\"ormander condition on $\Omega$. 
Since coherence is a local property, we can assume that $\Omega$ is a bounded domain in $\mathbb{C}^n$, $E=\Omega\times \underline{\mathbb{C}}^r$ is the trivial bundle over $\Omega$, and each element of $h^\star$ is bounded on $\Omega$. 
Let $H^0_{(2, h)}(\Omega, \underline{\mathbb{C}}^r)$ be the square integrable $\mathbb{C}^r$-valued holomorphic functions with respect to $h$ on $\Omega$. 
By the strong Noetherian property of coherent sheaves, $H^0_{(2, h)}(\Omega, \underline{\mathbb{C}}^r)$ generates a coherent ideal sheaf $\mathscr{F}\subset \mathscr{O}(E) = \mathscr{O}(\underline{\mathbb{C}}^r)$. 
First of all, we will show that 
\begin{equation*}
\mathscr{F}_x + \mathscr{E}(h)_x \cap \mathfrak{m}^{k+1}_x\cdot \mathscr{O}(\underline{\mathbb{C}}^r)_x = \mathscr{E}(h)_x, 
\end{equation*}
where $\mathfrak{m}_x$ is a maximal ideal of $\mathscr{O}_{\Omega, x}$ and $k$ is any positive integer. 
It is enough to show that 
\begin{equation}
\mathscr{E}(h)_x \subset \mathscr{F}_x + \mathscr{E}(h)_x \cap \mathfrak{m}^{k+1}_x\cdot \mathscr{O}(\underline{\mathbb{C}}^r)_x.\label{eq:coherent}
\end{equation} 
We take any element $f={}^t(f_{1,x}, \cdots, f_{r, x})\in \mathscr{E}(h)_x$, where $f_i$ is a holomorphic function defined in a neighborhood $U$ of $x$. 
Let $\theta$ be a cut-off fuction with support in $U$ such that $\theta=1$ in a neighborhood of $x$. We define a $\dbar$-closed $\mathbb{C}^r$-valued $(n, 1)$-form $\alpha$ as
\begin{equation*}
\alpha = \dbar(\theta f dz). 
\end{equation*}
We also take a weight function 
\begin{equation*}
\psi_\delta(z) = (n+k)\log (|z-x|^2+\delta^2)+|z|^2. 
\end{equation*}
Solving a $\dbar$-equation, we get a smooth $\mathbb{C}^r$-valued $(n, 0)$-form $u_\delta$ such that 
\begin{itemize}
		\item $\dbar u_\delta = \alpha$, and
		\item $\displaystyle\int_\Omega |u_\delta|^2_{(h, \omega)} e^{-\psi_\delta}  dV_\omega \leq \displaystyle\int_\Omega \sum_{1\le i, j \le n}(\psi_\delta^{(i\overline{j})}\alpha_i, \alpha_j)_{h} e^{-\psi_\delta}dV_\omega.$
	\end{itemize}
The right-hand side of the above inequality has an upper bound independent of $\delta$. 
Taking limits $\delta \to 0$ and repeating the argument of the proof of Theorem \ref{thm:conversevectorbundle}, 
we obtain a smooth $\mathbb{C}^r$-valued $(n, 0)$-form $udz$ such that 
\begin{itemize}
		\item $\dbar (udz)= \alpha$, and
		\item $\displaystyle\int_\Omega \frac{|u|^2_{h}}{|z-x|^{2(n+k)}}   dV_\omega < +\infty .$
	\end{itemize}
Since each element of $h^\star$ is bounded, there exists a positive constant $C$ such that 
\begin{equation*}
|g|^2_h \ge C |g|^2 = C (|g_1|^2+\cdots + |g_r|^2) 
\end{equation*}
for any $\mathbb{C}^r$-valued smooth function $g$. Hence we get 
\begin{equation*}
\int_\Omega \frac{|u_i|^2}{|z-x|^{2(n+k)}}   dV_\omega < +\infty. 
\end{equation*}
Letting $F:= \theta f-u$, we obtain $F\in H^0_{(2, h)}(\Omega, \underline{\mathbb{C}}^r)$ and $f_x - F_x = u_x\in \mathscr{E}(h)_x \cap \mathfrak{m}^{k+1}\cdot \mathscr{O}(\underline{\mathbb{C}}^r)$. 
This proves (\ref{eq:coherent}). 

Finally, we will prove $\mathscr{F}_x=\mathscr{E}(h)_x$. 
The Artin-Rees lemma implies that there exists an integer $l\ge 1 $ such that 
\begin{equation*}
\mathfrak{m}^{k+1}_x\cdot \mathscr{O}(\underline{\mathbb{C}}^r)_x \cap \mathscr{E}(h)_x = \mathfrak{m}_x^{k-l+1}(\mathfrak{m}^{l}_x\cdot \mathscr{O}(\underline{\mathbb{C}}^r)_x \cap \mathscr{E}(h)_x)
\end{equation*}
holds for any $k\ge l-1$. Therefore, we have 
\begin{align*}
\mathscr{E}(h)_x &= \mathscr{F}_x + \mathscr{E}(h)_x \cap \mathfrak{m}^{k+1}_x\cdot \mathscr{O}(\underline{\mathbb{C}}^r)_x \\
&=\mathscr{F}_x + \mathfrak{m}_x^{k-l+1}(\mathfrak{m}^{l}_x\cdot \mathscr{O}(\underline{\mathbb{C}}^r)_x \cap \mathscr{E}(h)_x) \\
& \subset \mathscr{F}_x + \mathfrak{m}_x\cdot \mathscr{E}(h)_x \\
& \subset \mathscr{E}(h)_x. 
\end{align*}
By Nakayama's lemma, we obtain $\mathscr{F}_x=\mathscr{E}(h)_x$. Thus we can conclude that $\mathscr{E}(h)$ is coherent. 
\end{proof}

Finally, we give a simple example of a singular Hermitian metric satisfying the twisted H\"ormander condition. 

\begin{theorem}\label{thm:examplehormander}
Let $h$ be a Griffiths semi-positive singular Hermitian metric. Then, $(E\otimes \det E, h\otimes \det h)$ is positively curved 
in the sense of twisted H\"ormander. 
\end{theorem}

\begin{proof}
We set $G:=E\otimes \det E$ and $g:=h\otimes \det h$ for simplicity. For any point $x\in X$, 
we take an open Stein neighborhood $U$ of $x$ such that $E$ is trivial over $U$. It is enough to show that 
$(G, g)$ satisfies the twisted H\"ormander condition on $U$. Let $m$ be a positive integer, $\psi$ be a smooth strictly plurisubharmonic function on $U$, 
and $\alpha=\sum_j \alpha_j dz\wedge d\overline{z}_j$ be a smooth $\dbar$-closed $G^{\otimes m}$-valued $(n, 1)$-form with compact support and finite norm $\int_U \sum_{1\le i, j \le n}(\psi^{(i\overline{j})}\alpha_i, \alpha_j)_{g^{\otimes m}} e^{-\psi}dV_\omega<+\infty$. 
We take a Stein exhaustion $\{ U_\lambda \}_{\lambda}$ of $U$ such that $U_\lambda \Subset U_{\lambda+1} \Subset \cdots \Subset U$ and each $U_\lambda$ is a Stein domain. 
It is known that there exists a sequence of smooth Hermitian metrics $\{ h_\nu \}_{\nu =1}^\infty$ with positive Griffiths curvature, increasing pointwise to $h$ on each $U_\lambda$ (\cite[Proposition 3.1]{BP}). 
From the results of Demailly-Skoda (\cite{DSk}), we see that $g_\nu$ is Nakano semi-positive for each $\nu$, where $g_\nu := h_\nu \otimes \det h_\nu$. 
Since $g^{\otimes m}_\nu e^{-\psi}$ is Nakano positive, we obtain a smooth $G^{\otimes m}$-valued $(n, 0)$-form $u_{\lambda, \nu}$ such that $\dbar u_{\lambda, \nu} = \alpha$ and 
\begin{align*}
\int_{U_\lambda} |u_{\lambda, \nu}|^2_{(g_\nu^{\otimes m}, \omega)}e^{-\psi} dV_\omega &\le \int_{U_\lambda} ([\sqrt{-1}\Theta_{g_\nu^{\otimes m}e^{-\psi}}, \Lambda_\omega ]^{-1}\alpha, \alpha)_{(g_\nu^{\otimes m}, \omega)} e^{-\psi}dV_\omega \\
& \le \int_{U_\lambda} \sum_{1\le i, j \le n}(\psi^{(i\overline{j})}\alpha_i, \alpha_j)_{g_\nu^{\otimes m}} e^{-\psi}dV_\omega \\
& \le \int_{U} \sum_{1\le i, j \le n}(\psi^{(i\overline{j})}\alpha_i, \alpha_j)_{g^{\otimes m}} e^{-\psi}dV_\omega 
\end{align*}
by using H\"ormander's $L^2$-estimate. 
Using a diagonal argument and taking weak limits $\lambda \to \infty, \nu \to \infty$, we get a smooth solution $u$ such that 
\begin{itemize}
		\item $\dbar u = \alpha$, and
		\item $\displaystyle\int_U |u|^2_{(g^{\otimes m}, \omega)} e^{-\psi}  dV_\omega \leq \displaystyle\int_U \sum_{1\le i, j \le n}(\psi^{(i\overline{j})}\alpha_i, \alpha_j)_{g^{\otimes m}} e^{-\psi}dV_\omega.$
	\end{itemize}
Therefore, we can conclude that $(G, g)$ satisfies the twisted H\"ormander condition on $U$. 
\end{proof}

Inspired by Theorem \ref{thm:coherent} and Example \ref{thm:examplehormander}, we propose the following question in relation to Question \ref{qus:griffithsnakano}. 

\begin{question}\label{que:nakanohormander}
Is positivity in the sense of twisted H\"ormander is more or less equivalent to Nakano positivity?
\end{question}
%%%%%%%%%%%%%%%%%%%%%%%%%%%%%%%%%%%%%%%%%%%%%%%%%%%%%%

%%%%%%%%%%%%%%%%%%%%%%%%%%%%%%%%%%%%%%%%%%%%%%%%%%%%%%
%%%%%%%%%%%%%%%%%%%%%%%%%%%%%%%%%%%%%%%%%%%%%%%%%%%%%%

\ \\
\vskip3mm
{\bf Acknowledgment. }
The authors would like to thank their supervisor Prof. Shigeharu Takayama for inspiring and helpful comments. 
This work is supported by the Program for Leading Graduate Schools, MEXT, Japan. 
This work is also supported by JSPS KAKENHI Grant Number 18J22119.

%%%%%%%%%%%%%%%%%%%%%%%%%%%%%%%%%%%%%%%%%%%%%%%%%%%%%%
%%%%%%%%%%%%%%%%%%%%%%%%%%%%%%%%%%%%%%%%%%%%%%%%%%%%%%

%%%%%%%%%%%% References %%%%%%%%%%%%%
%%
%<Author name> is written as Initial of Given Name, and Family Name.
%<Title> is written in roman letters.
%<Journal name> should be abbreviated according to
% the MR Serials Abbreviations List of Mathematical Reviews:
% (Abbreviations of Names of Serials; http://www.ams.org/mr-database)
%For <Pages>, use en-dash "--" between page numbers.
%%

\end{document}